\documentclass[tbtags]{article}
\usepackage{amsmath,amsfonts,amsbsy,amsthm,latexsym,amssymb,enumerate,graphicx, psfrag}

\newtheorem{theorem}{Theorem}

\newtheorem{proposition}[theorem]{Proposition}
\newtheorem{corollary}[theorem]{Corollary}

\theoremstyle{definition}
\newtheorem*{example*}{Example}
\newtheorem*{remark*}{Remark}

\newcommand\dto{\overset{\mathrm{d}}{\to}}
\newcommand\pto{\overset{\mathrm{p}}{\to}}
\newcommand\rhox{\nu}

\begin{document}
\sloppy

\title{The external lengths in Kingman's coalescent}
 
\author{ Svante Janson%
\thanks{Department of Mathematics, Uppsala University, PO Box 480,
SE-751~06 Uppsala, Sweden.
\texttt{svante.janson@math.uu.se}}
 \qquad
 G\"otz Kersting%
\thanks{Fachbereich Informatik und Mathematik, Universit\"at Frankfurt, Fach 
  187, D-60054 Frankfurt am Main, Germany.
\texttt{kersting@math.uni-frankfurt.de}} }
\date{17 January, 2011}  
  \maketitle
\begin{abstract}
In this paper we prove asymptotic normality of the total length of external
branches in Kingman's coalescent. The proof uses an embedded Markov chain,
which can be described as follows: Take an urn with $n$ {\em black}
balls. Empty it in $n$ steps according to the rule: In each step remove a
randomly chosen pair of balls and replace it by one {\em red} ball. Finally
remove the last remaining ball. Then the numbers $U_k$, $0 \le k \le n$, of
red balls after $k$ steps exhibit an unexpected property: $(U_0,
\ldots,U_n)$ and $(U_n, \ldots,U_0)$ are equal in distribution. 
\end{abstract}
\begin{small}
\end{small}

\section{Introduction and results}
Our main result in this paper is that the total length $L_n$ of all external branches  in Kingman's coalescent with $n$ external branches is asymptotically normal for $n \to \infty$. 

Kingman's coalescent (1982) consists of two components. First there are the coalescent times $T_1 >T_{2} > \cdots >T_n=0$. They are such that
\[ {k \choose 2}(T_{k-1}-T_{k})  \ , \quad k=2,\ldots,n \]
are independent, exponential random variables with expectation 1. Second
there are partitions $\pi_1=\big\{\{1,\ldots,n\}\big\}, \pi_2,\ldots ,
\pi_n=\big\{\{1\}, \ldots, \{n\}\big\}$ of the set $\{1,\ldots,n\}$, where
the set $\pi_k$ containes $k$ disjoint subsets  of $\{1,\ldots,n\}$ and
$\pi_{k-1}$ evolves from $\pi_k$  by merging two randomly chosen elements of
$\pi_k$. Moreover, $(T_n, \ldots, T_1)$ and $(\pi_n, \ldots, \pi_1)$ are
independent. For convenience we put $\pi_0:=\emptyset$.   

As is customary the coalescent can be represented by a tree with $n$ leaves labelled from 1 to $n$. Each of these leaves corresponds to an external branch of the tree. The other node of the branch with label $i$ is located at level
\[ \rho(i) := \max \{ k \ge 1 : \{i\} \not\in \pi_{k} \}  \]
within the coalescent.
The length of this branch  is  $T_{\rho(i)}$, 
The total external length of the coalescent is given by
\[ L_n := \sum_{i=1}^n T_{\rho(i)} \ . \]

This quantity is of a certain statistical interest. Coalescent trees have been introduced by Kingman as a model for the genealogic relationship of $n$ individuals, down to their most recent common ancestor. Mutations can be located everywhere on the branches. Then mutations on external branches affect only single individuals. This fact was used by Fu and Li (1993) in designing their $D$-statistic and providing a test whether or not data fit to Kingman's coalescent. 

Otherwise single external branches have mainly been studied in the
literature. The asymptotic distribution of $T_{\rho(i)}$ has been obtained
by Caliebe et al (2007), using a representation of its Laplace transform due
to Blum and Fran\c{c}ois (2005). We address this issue in Section \ref{S6}
below. Freund and M\"ohle (2009) investigated the external branch length of
the Bolthausen-Snitman coalescent, and Gnedin et al (2008) the
$\Lambda$-coalescent. 

Here is our main result.
\begin{theorem}\label{T1}
As $n \to \infty$,
\[\frac 12 \sqrt {\frac n{\log n}}\bigl(L_n - 2\bigr) \ \stackrel{d}{\to} \ N(0,1) \ .\]  
\end{theorem}

The proof will show that the limiting normal distribution originates from the random partitions and not from the exponential waiting times. 

A second glance on this result  reveals a peculiarity: The normalization of $L_n$ is carried out using its expectation, but only half of its variance. These two terms have been  determined by Fu and Li (1993) (with a correction given by Durrett (2002)).  They obtained 
\[ \mathbf E(L_n) = 2 \ , \quad \mathbf{Var}(L_n) = \frac{8n h_n -16n +8}{(n-1)(n-2)} \sim \frac{8 \log n}{n}  \]
with
$ h_n := 1+ \frac 12 + \cdots + \frac 1n $, the $n$-th harmonic number. Below we derive a more general result.

To uncover this peculiarity we shall study the external lengths in more
detail. First we look at the point processes $\eta_n$ on $(0,\infty)$, given
by $\eta_n=\sum_{i=1}^n\delta_{\sqrt n T_{\rho(i)}}$, 
i.e.
\begin{equation}\label{etan}
 \eta_n(B):= \# \{ i : \sqrt n T_{\rho(i)} \in B \}   
\end{equation}
for Borel sets $B \subseteq (0,\infty)$.

\begin{theorem}\label{T2}
As $n\to \infty$ the point process $\eta_n$ converges in distribution, 
as point processes on $(0,\infty]$, to a Poisson
point process $\eta$ on $(0,\infty)$ with intensity measure  
$\lambda(dx) = 8x^{-3} \, dx$.
\end{theorem}

We use $(0,\infty]$ in the statement of Theorem \ref{T2}
instead of $(0,\infty)$ since it is
  stronger, including for example $\eta_n(a,\infty)\dto \eta(a,\infty)$ for
  every $a>0$. The significance 
 is that, as $n\to\infty$, there will be points clustering at 0 but not at $\infty$.
(Below in the proof we recall the definition of convergence in distribution of point processes.)

Theorem \ref{T2} permits a first orientation. Since $\sqrt n L_n = \int x \,
\eta_n(dx)$, one is tempted to resort to infinitely divisible
distributions. However, the intensity measure $\lambda(dx)$ is slightly
outside the range of the L\'evy-Chintchin formula. Shortly speaking this
means that small points of $\eta_n$ have a dominant influence on the
distribution of $L_n$ and we are within the domain of the normal
distribution.

Thus  let  us look in more detail on the external lengths and focus on
\[ L_{n}^{\alpha,\beta} :=  \sum_{n^\alpha \le \rho(i) < n^\beta} T_{\rho(i)}\ , \quad 0 \le \alpha < \beta \le 1 \ , \]
which is the total length of those external branches having their internal nodes between level $\lceil n^\alpha\rceil$ and $\lceil n^\beta\rceil$ within the coalescent. Obviously $L_n=L_n^{0,1}$. 

\begin{proposition}\label{P3}
For $0 \le \alpha < \beta \le 1$
\[ \mathbf E(L_n^{\alpha,\beta}) = \frac{2}{n(n-1)}\bigl(\lceil n^\beta \rceil - \lceil n^\alpha \rceil \bigr)\bigl(2n+1-\lceil n^\beta \rceil - \lceil n^\alpha \rceil\bigr)
\]
and
\[\mathbf {Var}(L_n^{\alpha,\beta}) \sim 8(\beta-\alpha) \frac{\log n}{n} \ ,\]
as $n \to \infty$.  
\end{proposition}

In particular $\mathbf E(L_n^{1-\varepsilon,1})\sim \mathbf E(L_n^{0,1})$, whereas $\mathbf {Var}(L_n^{1-\varepsilon,1})\sim \varepsilon\mathbf {Var}(L_n^{0,1})$. Thus the proposition indicates that the systematic part of $L_n$ and its fluctuations  arise in different regions of the coalescent tree, the former close to the leaves and the latter closer to the root. 

Still this proposition gives an inadequate impression.

\begin{theorem}\label{T4}
For $0 \le \alpha < \beta < 1/2$
\[ \mathbf P(L_n^{\alpha,\beta} =0) \ \to\ 1 \]
as $n \to \infty$. 
Moreover 
\[ \sqrt{n} L_{n}^{0,\frac 12}\ \stackrel d{\to}\ \int_2^\infty x \, \eta(dx)  \]
and for $1/2 \le \alpha < \beta \le 1$
\[ \frac{L_n^{\alpha,\beta} - \mathbf E(L_n^{\alpha,\beta})}{\sqrt{\mathbf {Var}(L_n^{\alpha,\beta})}} \ \stackrel d{\to}\ N(0,1) \ .\]
In addition $L_{n}^{\alpha,\beta}$ and $L_{n}^{\gamma,\delta}$ are asymptotically independent for $\alpha < \beta \le\gamma< \delta$.  
\end{theorem}
This result implies Theorem \ref{T1}: In $L_n=L_{n}^{0,\frac 12}+ L_{n}^{\frac 12,1}$ the summands are of order $\sqrt {1/n}$ and $\sqrt{\log n/n}$, such that in the limit the second, asymptotically normal component  dominates. To this end, however, $n$ has to become exponentially large, otherwise the few long branches, which make up $ L_{n}^{0,\frac 12}$, cannot be neglected and may produce extraordinary large values of $L_n$. Thus the normal approximation for the distribution of $L_n$ seems little useful for practical purposes. One expects a fat right tail compared to the normal distribution. Indeed $\int_2^\infty x \, \eta(dx)$ has finite mean but infinite variance. 

This is illustrated by the following two histograms from 10000 values of $L_n$, where the length of the horizontal axis to the right indicates the range of the values.
\begin{center}
\psfrag{n1}{$n=50$}
\psfrag{n2}{$n=1000$}
\includegraphics[height=3cm,width=11cm]{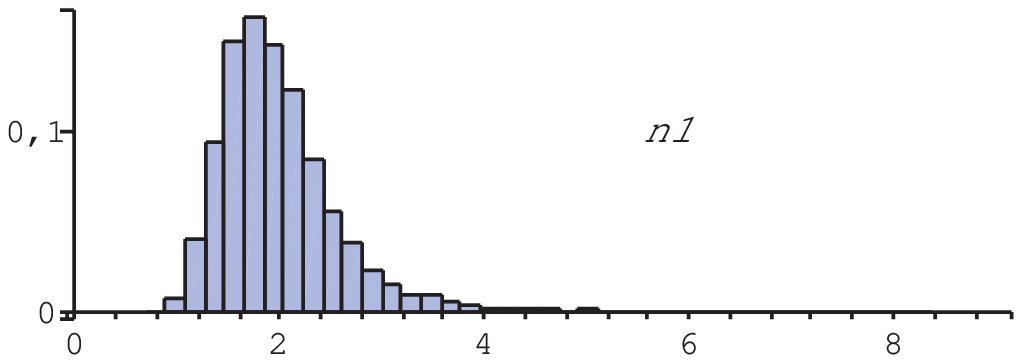}
\includegraphics[height=3cm,width=11cm]{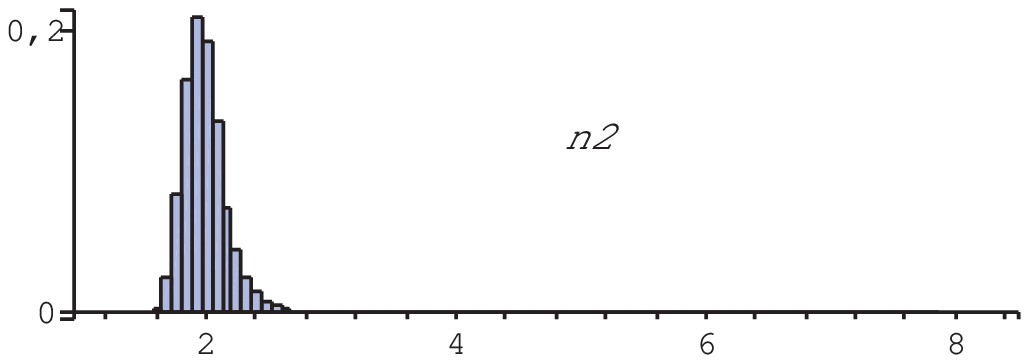}
\end{center}
The heavy tails to the right are clearly visible. Also very large outliers appear: For $n=50$ the simulated values of $L_n$ range from 0.685 to 8.38, and for $n=1000$ from 1.57 to 7.87.

Also it turns out that the approximation of the variance in Proposition
\ref{P3} is good only for very large $n$. This can be seen already from the
formula of Fu and Li.  To get an exact formula for the variance we look at a
somewhat different quantity, namely 
\[ \hat L_n^{\alpha,\beta} :=  \sum_{i=1}^n (T_{\rho(i)}\wedge T_{\lfloor n^\alpha\rfloor}-T_{\rho(i)}\wedge T_{\lfloor n^\beta\rfloor}) \]
with $0 \le \alpha < \beta \le 1$, which is the portion of the external
length between level $\lfloor n^\alpha\rfloor$ and $\lfloor n^\beta\rfloor$ 
within the coalescent. 

\begin{proposition}\label{P5}
For $0 \le \alpha \le 1$ with $m:=\lfloor n^\alpha\rfloor$
\[ \mathbf {E} (\hat L_n^{\alpha,1}) = 2 \frac{n-m}{n-1} \]
and
\[\mathbf {Var} (\hat L_n^{\alpha,1}) = \frac{8(h_{n-1}-h_{m-1})(n+2m-2)}{(n-1)(n-2)}-\frac{4(n-m)(4n+m-5)}{(n-1)^2(n-2)} \ . \]  
\end{proposition}
\noindent
For $\alpha =0$ we recover the formula of Fu and Li. A similar expression holds for $\hat L_n^{\alpha,\beta}$. 

Proposition \ref{P3} and Theorem \ref{T4} carry over to 
$\hat L_n^{\alpha,\beta}$, up to a change in expectation 
and with the limit
$\sqrt{n}\hat L_{n}^{0,\frac 12}\ \dto\ \int_2^\infty (x-2) \, \eta(dx)$. 
The following histogram
from a random sample of length 10000 shows that  already for $n=50$ the
distribution of $\hat L_n^{\frac 12,1}$ fits well to the normal distribution
when using the values for expectation and variance, given in Proposition
\ref{P5}.  
\begin{center}
\psfrag{n1}{$n=50$}
\includegraphics[height=3cm,width=11cm]{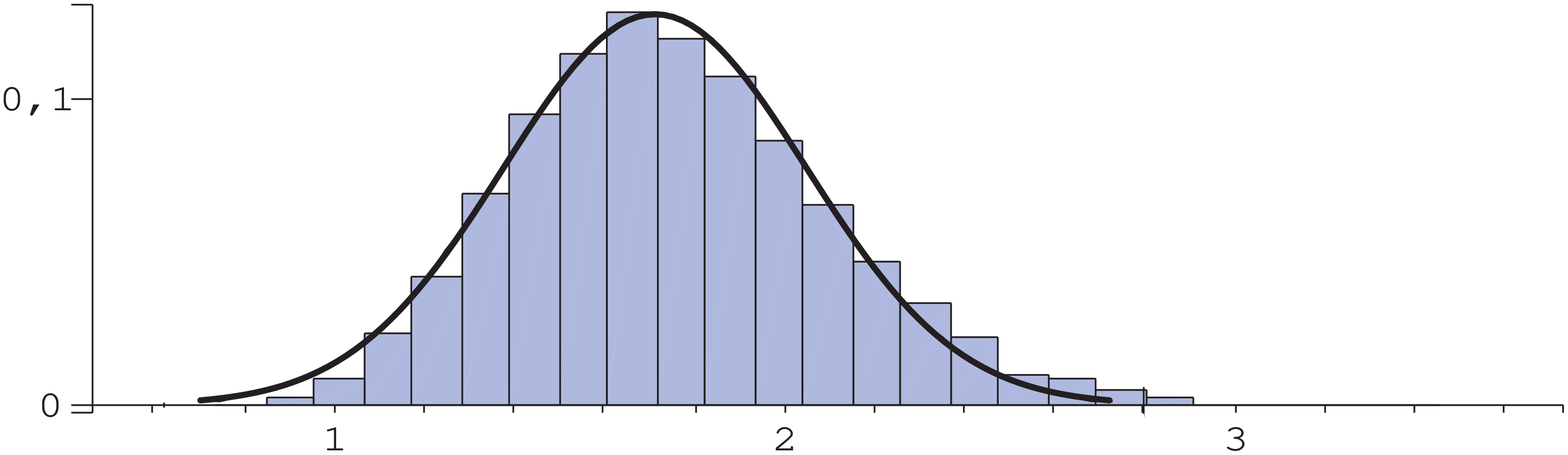}
\end{center}

Our main tool for the proofs is a representation of $L_n$ by means of an imbedded Markov chain $U_0,U_1,\ldots,U_n$, which is of interest of its own. We shall introduce it as an urn model. The relevant fact is that this model possesses an unexpected hidden symmetry, namely it is reversible in time.  This is our second main result. For the proof we use another urn model, which allows reversal of time in a simple manner.

The urn models are introduced and studied in Section \ref{S2}. Proposition
\ref{P3} is proven in Section \ref{S3}, Theorems \ref{T2} and \ref{T4} are
derived in Section \ref{S4} and Proposition \ref{P5} in Section \ref{S5}. In
Section \ref{S6} we complete the paper by considering the length of an
external branch chosen at random.

\section{The urn models}\label{S2}

Take an urn with $n$ {\em black} balls. Empty it in $n$ steps according to the rule: In each step remove a randomly chosen pair of balls and replace it by one {\em red} ball. In the last step remove the last remaining ball. Let 
\[ U_k := \text{ number of red balls in the urn after } k \text{ steps}\ . \]
Obviously $U_0=U_n=0$, $U_1=U_{n-1}=1$ and $1 \le U_k \le \min(k,n-k)$ for $2 \le k\le n-2$. $U_0, \ldots, U_n$ is a Markov chain with transition probabilities 
\[ \mathbf P( U_{k+1}=u' \mid U_k = u) = \begin{cases}\  {u \choose 2}\big/{n-k \choose 2} \ , & \text{if } u'=u-1 \ , \\ \ u(n-k-u)\big/{n-k \choose 2} \ , & \text{if } u'=u\ , \\ \ {n-k-u \choose 2}\big/{n-k \choose 2} \ , & \text{if } u'=u+1\ .
\end{cases} \]
We begin our study of the model by calculating expectations and covariances.
\begin{proposition}\label{P6}
For $0 \le k\le l \le n$
\[ \mathbf E(U_k) = \frac{k(n-k)}{n-1} \ , \quad \mathbf {Cov} (U_k,U_l) = \frac{k(k-1)(n-l)(n-l-1)}{(n-1)^2(n-2)} \ . \]  
\end{proposition}

\begin{proof}
Imagine that the black balls are numbered from 1 to $n$. Let $Z_{ik}$ be the indicator variable of the event that the black ball with number $i$ is not yet removed after $k$ steps. Then $U_k = n-k-\sum_{i=1}^n Z_{ik}$ and consequently
\[ \mathbf E(U_k) = n-k - n \mathbf E(Z_{1k}) \]
and for $k \le l$ in view of $Z_{1l} \le Z_{1k}$
\begin{align*}
\mathbf {Cov} (U_k,U_l) &= \sum_{i=1}^n \sum_{j=1}^n \mathbf{Cov}(Z_{ik},Z_{jl})\\
&= n(n-1) \mathbf E(Z_{1k} Z_{2l}) + n \mathbf E(Z_{1l})-n^2 \mathbf E(Z_{1k})\mathbf E(Z_{1l}) \ .
\end{align*} 
Also
\[ \mathbf P( Z_{1k}=1) = \frac{{n-1 \choose 2}}{{n \choose 2}} \cdots \frac{{n-k \choose 2}}{{n-k+1 \choose 2}}  = \frac{(n-k)(n-k-1)}{n(n-1)} \]
and for  $k \le l$
\begin{align*}
\mathbf P(Z_{1k}=1,Z_{2l}=1) &= \frac{{n-2 \choose 2}}{{n \choose 2}} \cdots \frac{{n-k-1 \choose 2}}{{n-k+1 \choose 2}} \cdot \frac{{n-k-1 \choose 2}}{{n-k \choose 2}} \cdots \frac{{n-l \choose 2}}{{n-l+1 \choose 2}} \\ &= \frac{(n-k-1)(n-k-2)(n-l)(n-l-1)}{n(n-1)^2(n-2)} \ .
\end{align*}
Our claim now follows by careful calculation. 
\end{proof}
\noindent
Note that these expressions for expectations and covariances are invariant under the transformation $k \mapsto n-k$, $l \mapsto n-l$. This is not by coincidence:
\begin{theorem}\label{T7}
$(U_0,U_1,\ldots,U_n)$  and $(U_n, U_{n-1},\ldots, U_0)$ are equal in
  distribution.   
\end{theorem}

\begin{proof}
Leaving aside $U_0=U_n=0$ we have  $U_k \ge 1$ $a.s.$ for the other values of $k$. Instead we shall  look at $U_k'= U_k-1$ for $1 \le k \le n-1$. It turns out that for this process one can specify a different dynamics, which is more lucid and amenable to reversing time.

Consider the following alternative box scheme: There are two boxes $A$ and $B$. At the beginning $A$ contains $n-1$ black balls whereas $B$ is empty. The balls are converted in $2n-2$ steps into $n-1$ red balls lying in $B$. Namely, in the steps number $1,3, \ldots, 2n-3$ a randomly drawn ball from $A$ is shifted to $B$ and in the steps number $2,4,\ldots, 2n-2$ a randomly chosen black ball (whether from $A$ or $B$) is recolored to a red ball. These $2n-2$ operations are carried out independently.

For $1 \le k \le n-1$ let
\begin{equation*}
U_k' := \text{number of red balls in box } A \text{ after } 2k-1 \text{ steps,}   
\end{equation*}
that is at the moment after the $k$th move and before the $k$th recoloring. Obviously the sequence is a Markov chain, also $U_1'=0$. 

As to the transition probabilities note that after $2k-1$ steps  there are   $n-k$ black balls in all and $n-k-1$ balls in $A$. Thus given $U_{k}'=r $ there are $r$ red and $n-k-r-1$ black balls in $A$, and the remaining $r+1$ black balls belong to $B$. Then $U_{k+1}'=r+1$ occurs only, if in the next step the ball recolored  from black to red belongs to $A$ and subsequently the ball shifted from $A$ to $B$ is black. Thus 
\begin{align*}
 \mathbf P(U_{k+1}'=r+1 \mid U_{k}'=r ) = \tfrac {n-k-r-1}{n-k} \cdot \tfrac{n-k-r-2}{n-k-1}= \tbinom {n-k-r-1}2/ \tbinom {n-k}2 \ .
\end{align*}
Similarly $U_{k+1}'=r-1$ occurs, if the recolored ball belongs to $B$ and next the ball shifted from $A$ to $B$ is red. The corresponding probability is
\begin{align*}
 \mathbf P(U_{k+1}'=r-1 \mid U_{k}'=r ) = \tfrac {r+1}{n-k} \cdot \tfrac{r}{n-k-1}= \tbinom {r+1}2/ \tbinom {n-k}2 \ .
\end{align*}
Since $U_1=1=U_1'+1$ and in view of the transition probabilities of $(U_k)$ and $(U_k')$ we see that $(U_1,\ldots, U_{n-1})$ and $(U_1'+1,\ldots,U_{n-1}'+1)$ indeed coincide in distribution.

Next note that $U_{n-1}'=0$. Therefore $U_k'$ can be considered as a function not only of the first $2k-1$  but also of the last $2n-2k-1$ shifting and recoloring steps. Since the steps are independent, the process backwards is equally easy to handle. Taking into account that backwards the order of moving and recoloring balls is interchanged, one may just repeat the calculations above to obtain reversibility. 

But this repetition can be avoided as well. Let us put our model more formally: Label the balls from $1$ to $n-1$ and write the state space as
\[ S:= \big\{ \bigl((L_1,c_1), \ldots, (L_{n-1},c_{n-1})\bigr) \mid  L_i \in \{ A,B\}, c_i \in \{b,r\} \big\} \ ,\]
where $L_i$ is the location of ball $i$ and $c_i$ its color. Then in our model the first and second coordinate are changed in turn from $A$ to $B$ and from $b$ to $r$. This is done completely at random, starting within the first coordinates. Clearly we may interchange the role of the first and second coordinate. Thus our box model is equivalent to the following version:

Again initially $A$ contains $n-1$ black balls whereas $B$ is empty. Now in the steps number $1,3, \ldots, 2n-3$ a randomly chosen black ball is recolored to a red ball and in the steps number $2,4,\ldots, 2n-2$ a randomly drawn ball from $A$ is shifted to $B$. Again these $2n-2$ operations are carried out independently. Here we consider
\[ U_k'' := \text{number of black balls in box } B \text{ after } 2k-1 \text{ steps.} \]
Then from the observed symmetry it is clear that $(U_1',\ldots,U_{n-1}')$ and $(U_1'',\ldots,U_{n-1}'')$ are equal in distribution.

If we finally interchange both colors and boxes as well, then we arrive at the dynamics of the backward process. This finishes the proof.
\end{proof}

\noindent
There is a variant of our proof, which  makes the reversibility of $(U_k')$
manifest in a different manner. Let again the balls be labelled from $1$ to
$n-1$. Denote
\begin{align*}
\rhox_m &:= \text{instance between } 1 \text{ and } n-1, \text{ when ball } m \text{ is colored to red},\\
\sigma_m &:= \text{instance between } 1 \text{ and } n-1, \text{ when ball } m \text{ is shifted to box }B.
\end{align*}
Then from our construction it is clear that $\rhox=(\rhox_m)$ and
$\sigma=(\sigma_m)$ are two independent random permutations of the 
numbers $\{1,\dots,n-1\}$. 
Moreover, at instance $k$ (i.e. after $2k-1$ steps) ball number
$m$ is red and belongs to box $A$, if it was colored before and shifted
afterwards, i.e. $\rhox_m < k < \sigma_m$. Thus we obtain the formula 
\begin{equation}\label{uk'}
 U_k' = \# \{ 1 \le m \le n-1 : \rhox_m < k < \sigma_m \}   
\end{equation}
and we may conclude the following result.

\begin{corollary}\label{C8}
Let $\rhox$ and $\sigma$ be two independent
  random permutations of  $\{1,\dots,n-1\}$. 
Then $(U_1, \ldots, U_{n-1})$ is
  equal in distribution to the process
\[ \bigl(  \# \{ 1 \le m \le n-1 : \rhox_m < k < \sigma_m \} +1\bigr)_{1 \le k \le n-1} \ .\]  
\end{corollary}

Certainly this representation implies Theorem \ref{T7} again. Also it contains
additional information. For example, 
it is immediate that $U_k-1$ has a hypergeometric distribution with parameters $n-1, k-1,n-k-1$. 

The next example contains a first application of Theorem \ref{T7} to our original urn model.

\begin{example*}
Let us consider 
$\tau_n= \max\{k \ge 1: U_{n-k}=k\}$, 
the number of red balls in the urn, after the last black ball
has been removed. From reversibility $\tau_n$ has the same distribution as
the moment $\tau_n'=\max\{k\ge 1: U_{k}=k\}$, 
before the first red
ball is taken away from the urn. Thus 
\[ \mathbf P( \tau_n \ge k) = \frac{{n-2 \choose 2}}{{n-1\choose 2}} \frac{{n-4\choose 2}}{{n-2\choose 2}}\cdots \frac{{n-2k+2\choose 2}}{{n-k+1\choose 2}} = \frac{(n-k)\cdots (n-2k+1)}{(n-1)\cdots(n-k)}\ . \]
It follows for $t \ge 0$ 
\[ \mathbf P\Bigl(\frac{\tau_n}{\sqrt n} \ge t\Bigr) \to \exp(-t^2) \ , \]
as $n \to \infty$. \qed  
\end{example*}

More generally the dynamics of our urn looks as follows: Clearly, if $n$ is
large, then in the beginning always two black balls are removed from the
urn. The rare  
moments, when red balls are taken away, appear with increasing rate. Indeed it is not difficult to see that in the limit $n \to \infty$ and after a $\sqrt n$-scaling of time these instances build up a Poisson process with linearly increasing rate. As we have seen the picture remains the same after reversal of time. This will be made more precise in Section \ref{S4}.\\\\
We conclude this section by imbedding our urn model into the coalescent. Let
\begin{equation}\label{Vk}
   V_k :=k- \# \{ i  :\rho(i) < k \} \ , 
\end{equation}
and $U_{k}:=V_{n-k}$, $0 \le k \le n$. 
Thus $V_k$ is the number of internal branches among the $k$ branches after
the $(n-k)$-th coalescing event and $U_k$ is the number of internal branches among the $n-k$ branches after the $k$-th coalescing event. 
The coalescing mechanism takes two random branches and combines them into
one internal branch. If we code the external branches by black balls and the
internal branches by red, 
this completely conforms to our urn model;
thus $(U_0,\dots,U_n)$ is as above. By Theorem \ref{T7}, $(V_0,\dots,V_n)$ has the same distribution as
$(U_0,\dots,U_n)$. 
In the next sections we make use
of the Markov chain $V_0,\ldots,V_n$ and its properties. 

\begin{remark*}
  For a different interpretation of the process $(U_k)$, suppose that we
  have $n-1$ pairs of (different) shoes, and that all left shoes are mixed
  in one pile and all right shoes in another. We sort the shoes by taking
first a left shoe (at random), then a right shoe (also at random), then
another left shoe, and so on.
As soon as we take a shoe that matches one that we already have picked, we
put away the pair; otherwise we put the shoe on the table in front of us. 
If the pairs are numbered and $\rhox_m$ is the time right shoe 
$m$ is picked, and $\sigma_m$ the time left shoe $m$ is picked, 
then right shoe $m$ is on the table when 
the $k$-th left shoe has been picked if and only if $\rhox_m<k<\sigma_m$,
so by \eqref{uk'},  
the number of right shoes remaining on the table when the $k$-th left shoe
has been picked is $U_k'$, $1\le k\le n-1$.
The number of left shoes remaining on the table at the same time is
$U_k'+1=U_k$, so the total number of shoes on the table is $2U_k-1$.

This is a variation of the sock-sorting process studied in 
Steinsaltz (1999) and Janson (2009), Section 8,
which is similar except that there is
no difference between left and right; we obtain it if we mix all shoes in
one pile and pick from it at random.
(See Janson (2009) for other interpretations, including \emph{priority
  queues}, and further references.) 
It is not surprising that we have the same asymptotical behaviour of
$U_k$ and $\max_k U_k$ as for the sock-sorting problem. In particular,
we mention the following Gaussian process limit result, cf.\ Theorem 8.2
in Janson (2009). (This result is not used in the sequel.)
\end{remark*}

\begin{theorem}
As $n\to\infty$,
  the stochastic process
$n^{-1/2}\bigl(U_{\lfloor nt\rfloor}-nt(1-t)\bigr)$ converges in $D[0,1]$ to
  a continuous Gaussian process $Z(t)$ with mean $\mathbf E (Z(t))=0$ and
covariance function 
\begin{equation*}
  \mathbf E\bigl(Z(s)Z(t)\bigr) = s^2(1-t)^2, 
\qquad 0\le s\le t\le1.
\end{equation*}
\end{theorem}

\begin{proof}[Sketch of proof]
  Note first that $\mathbf E (U_{\lfloor nt\rfloor}) = nt(1-t)+O(1)$ by 
Proposition \ref{P6}.

It is easily seen that
\begin{equation*}
  \mathbf E ( U_{k+1}\mid U_k)
= U_k - \frac2{n-k}U_k+1
=  \frac{n-k-2}{n-k}U_k+1
\end{equation*}
and it follows that 
\begin{equation*}
M_k:=
\frac{U_k-\mathbf E( U_k)}{(n-k)(n-k-1)}
=
\frac{U_k}{(n-k)(n-k-1)}  - \frac{k}{(n-1)(n-k-1)},  
\end{equation*}
$k=0,1,\dots,n-2$, is a martingale.

Consider in the sequel only $k\le(1-\delta)n$ for some fixed $\delta>0$.
Then $\mathbf{Var}(M_k)\le(n-k-1)^{-4} \mathbf{Var}(U_k) = O(n^{-3})$, and it
follows from Doob's inequality that 
\begin{equation*}
  \max_k |U_k - \mathbf E (U_k)| = O_P(n^{1/2}).
\end{equation*}
(Using Theorem \ref{T7} we see that this extends to $0\le k\le n$.)
A straightforward computation of the conditional quadratic variation 
$\langle M,M\rangle_m 
:= \sum_{k<m} \mathbf E\bigl( (M_{k+1}-M_k)^2\mid U_k)$
shows that, uniformly in $0\le t\le1-\delta$,
\begin{equation*}
n^3  \langle M,M\rangle_{\lfloor{nt}\rfloor} \pto \frac{t^2}{(1-t)^2},
\end{equation*}
which implies, 
see Theorem VIII.3.11 in Jacod and  Shiryaev (1987),
that $n^{3/2}M_{\lfloor{nt}\rfloor}\dto \hat Z(t)$ in $D[0,1-\delta]$, 
where $\hat Z(t)$ is a Gaussian martingale given by $\hat
Z(t) = W (t^2/(1-t)^2)$ for a standard Brownian motion $W(t)$.
The result follows, for $t\in [0,1-\delta]$, with $Z(t)=(1-t)^2 \hat Z(t)$.

Since $\delta>0$ is arbitrary, this yields convergence in $D[0,1)$. By
  time-reversal and Theorem \ref{T7}, we also have convergence in $D(0,1]$,
and together these imply convergence in $D[0,1]$, see e.g.\
the proof in Janson (2009).
\end{proof}

\section{Proof of Proposition \ref{P3}}\label{S3}

We use the representation
\[ L_n^{\alpha,\beta} = \sum_{n^\alpha \le k < n^\beta} T_k X_k \ , \]
where  
\[ X_k := \#\{ i :\rho(i)=k \} \ , \]
$1 \le k < n$. 
In view of the coalescing 
procedure $X_k$ takes only the values $0,1,2$,
and
from the definition \eqref{Vk} of $V_k$
\begin{equation}\label{Xk}
 X_k =1+V_{k}-V_{k+1}  \ .   
\end{equation}
From \eqref{Xk}, $V_k=U_{n-k}$ and Proposition \ref{P6} 
we obtain
after simple calculations 
\begin{align} \mathbf E(X_k) = \frac {2k}{n-1} \ , \quad \mathbf{Var}(X_k) =\frac{2k(n-k-1)(n-3)}{(n-1)^2(n-2)} 
\label{erwartung}
\end{align}
and for $k < l$
\begin{equation}\label{covX}
 \mathbf {Cov} (X_k,X_l) = -\frac{4k(n-l-1)}{(n-1)^2(n-2)} \ .  
\end{equation}
Also from $T_k = \sum_{j=k+1}^n (T_{j-1}-T_j)$ we have
$\mathbf E(T_k) = 2\sum_{j=k+1}^n \frac {1}{(j-1)j}$ and $\mathbf {Var}(T_k) = 4 \sum_{j=k+1}^n \frac {1}{(j-1)^2j^2}$; thus
\begin{align} \mathbf E(T_k) =2\Bigl( \frac 1k- \frac 1n\Bigr) \ , \quad \mathbf {Var}(T_k) \le \frac c{k^3}  \label{terwartung}
\end{align}
for a suitable $c>0$, independent of $n$. 

Thus from independence
\[ \mathbf E(L_n^{\alpha,\beta}) = \sum_{n^\alpha \le k < n^\beta} 2\Bigl(\frac 1k- \frac 1n\Bigr) \frac {2k}{n-1} \ . \]
Now the first claim follows by simple computation.

Further from independence
\begin{equation}\label{var1}
\mathbf {Var} \Bigl(\sum_{n^\alpha \le k < n^\beta} (T_k-\mathbf E(T_k))
X_k\Bigr) 
= \sum_{n^\alpha \le k,l < n^\beta} \mathbf {Cov}(T_k,T_l) \mathbf E(X_kX_l) 
\ .   
\end{equation}
Using \eqref{erwartung}--\eqref{terwartung} we have  for $k < l$, 
\begin{equation*}
\mathbf {Cov}(T_k,T_l)\mathbf E(X_kX_l) 
= \mathbf {Var}(T_l)\mathbf E(X_kX_l)
\le \mathbf {Var}(T_l)\mathbf E(X_k)\mathbf E(X_l)
\le  \frac{c}{l^3}\cdot\frac{4kl}{(n-1)^2},
\end{equation*}
and
it follows that
\begin{equation*}
  \begin{split}
0\le
\sum_{n^\alpha \le k<l < n^\beta} \mathbf {Cov}(T_k,T_l) \mathbf E(X_kX_l) 	
&\le
\sum_{n^\alpha \le k<l < n^\beta} \frac{4ck}{l^2}(n-1)^{-2}
\\
&\le
\sum_{n^\alpha \le k < n^\beta} 4c(n-1)^{-2}
   = O(n^{-1}) \ .
  \end{split}
\end{equation*}
Consequently, \eqref{var1} yields,
using again \eqref{erwartung}--\eqref{terwartung},
\begin{equation}\label{variance}  
  \begin{split}
\mathbf {Var} \Bigl(\sum_{n^\alpha \le k < n^\beta} &(T_k-\mathbf E(T_k))
X_k\Bigr) 
= 
\sum_{n^\alpha \le k < n^\beta}  \mathbf {Var}(T_k) \mathbf E(X_k^2)+O(n^{-1})
 \\&
\le c \sum_{n^\alpha \le k < n^\beta} \frac 1{k^3} \Bigl(\frac{2k}{n-1}+
\frac{4k^2}{(n-1)^2} \Bigr) 
+O(n^{-1}) 
\\&
\le \frac{6c}{n-1} \sum_{n^\alpha \le k < n^\beta} \frac 1{k^2}  
+O(n^{-1}) 
=
O(n^{-1}) \ .	
  \end{split}
\end{equation}

It remains to show that
\[\mathbf {Var} \Bigl(\sum_{n^\alpha \le k < n^\beta} \mathbf E(T_k) X_k\Bigr) \sim  8(\beta-\alpha) \frac{\log n}{n} \ . \]
Now 
\begin{align*} \Big|\sum_{n^\alpha \le k < l < n^\beta}&\mathbf E(T_k)\mathbf E(T_l) \mathbf {Cov}(X_k,X_l)\Big| \\
&\le \sum_{n^\alpha \le k < l < n^\beta} \frac{2}{k}\cdot \frac{2}{l}\cdot \frac{4k}{(n-1)^2} = 16 \sum_{n^\alpha < l < n^\beta}\frac{l-\lceil n^\alpha\rceil }{l(n-1)^2}  = O(n^{-1})
\end{align*}
and consequently 
\begin{align*}\mathbf {Var} &\Bigl(\sum_{n^\alpha \le k < n^\beta} \mathbf E(T_k) X_k\Bigr) \\&= \sum_{n^\alpha \le k < n^\beta} \mathbf E(T_k)^2 \mathbf {Var} (X_k) + O(n^{-1})\\
&= \sum_{n^\alpha \le k < n^\beta} \frac{4}{k^2} \cdot \frac{2k}{n} 
\Bigl(1+O\Bigl(\frac kn\Bigr)\Bigr)
+  O(n^{-1}) 
=  8(\beta-\alpha) \frac{\log n}{n} 
+  O(n^{-1}) \ . 
\end{align*}
This gives our claim.

\section{Proof of Theorems \ref{T2} and \ref{T4}}\label{S4}

In this section we use Theorem \ref{T7}. Namely, $V_0,\ldots, V_n$ is a Markov chain  with transition probabilities, which can be  expressed by means of $X_1,\ldots,X_{n-1}$ as follows:
\begin{align*}
\mathbf P(X_k=x\mid V_k=v)= \begin{cases} {n-k-v \choose 2}/ {n-k\choose 2} \ , &\text{if } x=0\ ,\\ v(n-k-v)/{n-k \choose 2} \ , &\text{if } x=1\ ,\\ {v\choose 2}/{n-k \choose 2} \ , &\text{if } x=2 \ .
\end{cases}
\end{align*}
We like to couple these random variables with suitable independent random variables taking values 0 or 1. Note that $V_k$ takes only values $v \le k$, thus for $k \le n/3$
\[ {n-k-v \choose 2}\Big/ {n-k\choose 2} \ge {n-2k \choose 2}\Big/ {n-k\choose 2} \ge \frac{n-3k }{n-k} \ . \]
Therefore we may enlarge our model by means of random variables $Y_k$, $k \le n/3$, such that
 \begin{align*}
\mathbf P(X_k=x,\ &Y_k=y\mid V_k=v, V_{k-1},\ldots,V_0,Y_{k-1},\ldots,Y_1)\\&= \begin{cases} \frac{n-3k}{n-k} \ , &\text{if } x=0,y=0\ , \\ {n-k-v \choose 2}/ {n-k\choose 2}- \frac{n-3k}{n-k}\ , &\text{if } x=0,y=1\ ,\\v(n-k-v)/{n-k \choose 2} \ , &\text{if } x=1,y=1\ ,\\ {v\choose 2}/{n-k \choose 2} \ , &\text{if } x=2,y=1 \ .
\end{cases}
\end{align*}
For $\mathbf P(X_k=x\mid V_k=v)$ this gives the above formula, whereas
\begin{align*} \mathbf P( Y_k=y \mid V_k = v, V_{k-1},\ldots,V_0,Y_{k-1},\ldots,Y_1)= \begin{cases}\  \frac{n-3k}{n-k} \ , &\text{if } y=0 \ , \\
\ \frac{2k}{n-k} \ , & \text{if } y=1\ .\\
\end{cases}
\end{align*}
This means that the  0/1-valued random variables $Y_k$, $k \le n/3$, are
independent. For convenience we put $Y_k=0$ for $k>n/3$. A straightforward
computation gives 
\begin{align}\label{ey-x}
\mathbf E(Y_{k}-X_{k}\mid  V_k = v)&= \frac{2(k-v)}{n-k} \ ,\\
\mathbf E((Y_{k}-X_{k})^2\mid  V_k=v) &= \frac{2(k-v)}{n-k}+
\frac{2v(v-1)}{(n-k)(n-k-1)} \notag \\ 
&\le \frac{2(k-v)}{n-k}+ \frac{2k(k-1)}{(n-k)(n-k-1)}
\end{align}
for $k\le n/3$. Since $k-\mathbf E(V_k)=k(k-1)/(n-1)$ from Proposition \ref{P6}, it follows
\begin{align}
\mathbf E((Y_k-X_k)^2) \le \frac{4k(k-1)}{(n-k)(n-k-1)} \ .  \label{pointpr}
\end{align}

\begin{proof}[Proof of Theorem \ref{T2}] 
Recall that, by \eqref{etan} and \eqref{Xk}, 
\begin{equation}\label{etan2}
\eta_n=\sum_{i=1}^n\delta_{\sqrt n T_{\rho(i)}}
=\sum_{k=1}^{n-1} X_k \delta_{\sqrt n T_{k}}.  
\end{equation}

Recall also that $\eta_n\dto\eta$ as point processes on the interval
$(0,\infty]$ 
means that $\int f\,d\eta_n\dto\int f\,d\eta$ for every continuous $f$ with
compact support in $(0,\infty]$, or equivalently $\eta_n(B)\dto\eta(B)$ for
every 
relatively compact Borel subset $B$ of $(0,\infty]$ such that $\eta(\partial
B)=0$ a.s. (Here $B$ is relatively compact, if $B \subseteq [\delta,\infty]$ 
for some $\delta > 0$.) 
See, for example, the Appendix in Janson and Spencer (2007) 
and Chapter 16 (in particular Theorem 16.16) in Kallenberg (2002).

Let us first look at the point process
\begin{equation}\label{eten}
 \eta_n' := \sum_{k=1}^{n-1} Y_k \delta_{2\sqrt n/k} \ .  
\end{equation}
For $0<a<b\le \infty$
\[ \eta_n'([a,b)) = \sum_{ \frac{2\sqrt n}b < k \le \frac{2\sqrt n}a} Y_k  \]
and 
\[ \mathbf E\bigl(\eta_n'([a,b))\bigr) = \sum_{ \frac{2\sqrt n}b < k \le \frac{2\sqrt n}a} \frac{2k}{n-k} \to 4(a^{-2}-b^{-2})= 8 \int_a^b \frac{dx}{x^3} \ , \]
thus we obtain from standard results on sums of independent 0/1-valued
random variables that $\eta_n'([a,b))$ has asymptotically a Poisson
  distribution. Also $\eta_n'(B_1), \ldots,\eta_n'(B_i)$ are independent for
  disjoint $B_1, \ldots,B_i$. Therefore we obtain from standard results on
  point processes 
(for example Kallenberg (2002), Proposition 16.17) 
weak convergence of $\eta_n'$ to the Poisson point process
  $\eta$ on $(0,\infty]$ with intensity $8x^{-3}\, dx$.    

Next we prove that for all $0< a< b \le\infty$
\[ \eta_n([a,b))-\eta_n'([a,b)) \to 0 \]
in probability. To this end note that from \eqref{pointpr}
\[ \mathbf E\Big[\sum_{ k \le \frac{2\sqrt n}a}  (Y_k-X_k)^2 \Big] = O(n^{-1/2}) \ , \] 
which implies that $\mathbf P(X_k=Y_k \text{ for all }k \le \frac{2\sqrt n}a)\to 1$.  Therefore we may well replace $Y_k$ by $X_k$ in $\eta_n'([a,b))$.

Also, by \eqref{terwartung},  
$\sqrt nT_k -2\sqrt{n}/k = \sqrt nT_k- \sqrt n\mathbf E(T_k)  -2/\sqrt{n}$.
From \eqref{terwartung} and Doob's inequality for any $ \varepsilon >0$
\[ \mathbf P\bigl( \max_{ k \ge n^{2/5}} \sqrt n| T_k - \mathbf E(T_k)| \ge \varepsilon\bigr) \le  \frac n{\varepsilon^2} \mathbf {Var} (T_{\lceil n^{2/5} \rceil}) =O(n^{-1/5})\ . \]
Since $\mathbf P(Y_k=0 \text{ for all } k< n^{2/5}) \to 1$, we may as well
also replace $ 2\sqrt{n}/k$ by $\sqrt nT_k$ in $\eta_n'$,
which yields  $\eta_n$ by \eqref{etan2} and \eqref{eten} 
(use for example Kallenberg (2002), Theorem 16.16).
Thus the proof of
Theorem \ref{T2} is complete.  
\end{proof}

\begin{proof}[Proof of Theorem \ref{T4}]
As to the first claim of Theorem \ref{T4} observe that the events
$\{L_n^{0,\beta}=0\}=\{X_k=0 \text{ for all } k < n^\beta\}$ and
$\{V_{\lceil n^\beta\rceil}=\lceil n^\beta\rceil\}$ are equal. Thus 
\begin{equation*}
  \begin{split}
\mathbf P(L_n^{\alpha,\beta}> 0) 
&\le \mathbf P(L_n^{0,\beta}> 0) 
= \mathbf P(\lceil n^\beta\rceil-V_{\lceil n^\beta\rceil}\ge 1) 
\\&
\le \mathbf E(\lceil n^\beta\rceil-V_{\lceil n^\beta\rceil}) 
= \frac{\lceil n^\beta\rceil(\lceil n^\beta\rceil-1)}{n-1} \ .	
  \end{split}
\end{equation*}
For $\beta < 1/2$ this quantity converges to zero, which gives the first claim of the theorem.

For the next claim we use that because of \eqref{terwartung} $\sqrt
nT_{\lceil n^{1/2}\rceil}$ has expectation  $2 + O(n^{-1/2})$ and variance
of order $n^{-1/2}$. Thus $\mathbf P(2-\varepsilon < \sqrt nT_{\lceil
  n^{1/2}\rceil} < 2 + \varepsilon)\to 1$ for all $\varepsilon > 0$. This
implies that the probability of the event 
\begin{align*} \int_{[2+\varepsilon, \infty)} &x\, \eta_n(dx)
= \sqrt	n\sum_{k=1}^n T_kX_k I_{\{\sqrt n T_k \ge 2+ \varepsilon\}}\\ 
&\le \sqrt n\sum_{k < \sqrt n} T_kX_k= \sqrt nL_n^{0,\frac 12} \\
&\le \sqrt n\sum_{k=1}^n T_kX_k I_{\{\sqrt n T_k \ge 2- \varepsilon\}}
=\int_{[2-\varepsilon, \infty)} x\, \eta_n(dx)
\end{align*}
goes to 1. Also for $a >0$ from Theorem \ref{T2} $\int_a^\infty x \,
\eta_n(dx) \to \int_a^\infty x \, \eta(dx)$ in distribution. Altogether we
obtain,
letting $\epsilon\to0$,
\[ \sqrt nL_n^{0,\frac 12} \to \int_2^\infty x \, \eta(dx) \ ,\]
which is  our second claim.

As to the last claim of Theorem \ref{T4} we note that from \eqref{variance}  
\begin{equation}\label{mag}
 L_n^{\alpha,\beta} 
= \sum_{n^\alpha \le k < n^\beta} \mathbf E(T_k)X_k +O_P(n^{-1/2})   
\end{equation}
in probability, and also in $L^1$. 
In this representation we like to replace $X_k$ by $Y_k$. 
We assume first $\beta<1$. 
Note that for $\beta < 1$ in view of \eqref{terwartung} and \eqref{pointpr}
\begin{align*}
\mathbf {Var} \Bigl( \sum_{n^\alpha \le k < n^\beta}&\mathbf E(T_k)(Y_k-X_k -\mathbf E(Y_k-X_k \mid V_k)) \Bigr)\\ &\le \sum_{n^\alpha \le k < n^\beta} \frac 4{k^2}\mathbf E((Y_k-X_k)^2) = O(n^{\beta-2}) 
\end{align*}
and from \eqref{ey-x}, \eqref{terwartung} and Proposition \ref{P6} 
\begin{align*}
\mathbf {Var} \Bigl( \sum_{n^\alpha \le k < n^\beta}&\mathbf E(T_k)\mathbf E(Y_k-X_k\mid V_k) \Bigr)= \mathbf {Var} \Bigl( \sum_{n^\alpha \le k < n^\beta}\mathbf E(T_k)\frac{2V_k}{n-k} \Bigr)\\
&\le 2 \sum_{n^\alpha\le k \le l < n^\beta} 
 4\frac{\mathbf E(T_k)\mathbf E(T_l)}{(n-k)(n-l)} \mathbf {Cov}(V_k,V_l) \\
&\le 32 \sum_{n^\alpha\le k \le l < n^\beta} \frac kl \cdot
 \frac{(n-l)}{(n-k)(n-1)^2(n-2)} = O(n^{2\beta-3})
\ .
\end{align*}
Thus
$ \sum_{n^\alpha \le k < n^\beta} \mathbf E(T_k)
 \bigl(Y_k-X_k)-\mathbf E (Y_k-X_k)\bigr) 
=O_P(n^{-1/2})$
and \eqref{mag} yields
\[
L_n^{\alpha,\beta} - \mathbf E(L_n^{\alpha,\beta} ) =  \sum_{n^{\alpha}\le k < n^\beta} \mathbf E(T_k) (Y_k -\mathbf E(Y_k)) + O_P(n^{-1/2})\ .  
\]
Also 
$\mathbf {Var} (\frac 1n\sum_{n^{\alpha}\le k < n^\beta} Y_k) \le n^{-2} \sum_{n^{\alpha}\le k < n^\beta} 2k/(n-k) = O(n^{-1})$, and because of \eqref{terwartung} we end up with
\begin{align}
 L_n^{\alpha,\beta} - \mathbf E(L_n^{\alpha,\beta} ) =  2\sum_{n^{\alpha}\le k < n^\beta} \frac{Y_k -\mathbf E(Y_k)}{k} + O_P(n^{-1/2})\ . 
\label{darstellung}
\end{align}
This is a representation of the external length by a sum of independent random variables. 

Now $\mathbf {Var}(Y_k) = \frac{2k}{n-k}-\frac{4k^2}{(n-k)^2}$, thus for $\beta < 1$
\begin{align*} \mathbf {Var} \Bigl(2\sum_{n^{\alpha}\le k < n^\beta} \frac{Y_k -\mathbf E(Y_k)}{k}\Bigr) &= 4\sum_{n^{\alpha}\le k < n^\beta} \Bigl(\frac{2}{k(n-k)} - \frac{4}{(n-k)^2}\Bigr)\\
&\sim  8(\beta-\alpha)\frac{\log n}{n} \ .
\end{align*}
Moreover for $\delta >0$ we have $\mathbf E(|Y_k -\mathbf E(Y_k)|^{2+\delta}) \le \frac{2k}{n-k} + (\frac{2k}{n-k} )^{2+\delta} \le \frac{4k}{n-k}$ for $k \le n/3$, thus
\[ \sum_{n^{\alpha}\le k < n^\beta} \frac{1}{k^{2+\delta}} \mathbf E(|Y_k -\mathbf E(Y_k)|^{2+\delta}) \le 4 \sum_{n^{\alpha}\le k < n^\beta} \frac{1}{k^{1+\delta}(n-k)} \le \frac{8}{\delta n} \frac{1}{(n^{\alpha}-1)^\delta} \ . \]
Thus for $\alpha \ge 1/2$ we get
\[ \sum_{n^{\alpha}\le k < n^\beta} \frac{1}{k^{2+\delta}} \mathbf E(|Y_k-\mathbf E(Y_k)|^{2+\delta})  = o\Bigl(\frac{(\log n)^{1+\delta/2}}{n^{1+\delta/2}}\Bigr) \ , \]
and we may use Lyapunov's criterion for the central limit theorem. 
Consequently, 
\eqref{darstellung} implies
\begin{equation*}
\frac{ L_n^{\alpha,\beta} - \mathbf E(L_n^{\alpha,\beta} )}
{\sqrt{8(\beta-\alpha)\log n/n}} 
\dto N(0,1).  
\end{equation*}
This finishes the proof in the case $\beta < 1$,
using Proposition \ref{P3}.  

The case $\beta=1$ then
follows from $L_n^{\alpha,1}= L_n^{\alpha,1-\varepsilon}+
L_n^{1-\varepsilon,1}$ using Proposition~\ref{P3}. 

The last claim on asymptotic independence follows from \eqref{darstellung},
too.   
\end{proof}

\section{Proof of Proposition \ref{P5}} \label{S5}
Let $0\le\alpha \le 1$ and $m=\lfloor n^\alpha \rfloor$. Since $k-V_k = \# \{ i  :\rho(i) < k \}$ is the number of external branches, which are found between level $k-1$ and $k$, 
\[ \hat L_n^{\alpha,1} =  \sum_{m < k \le n} (T_{k-1}-T_k)(k-V_k) \ . \]
From independence
\[ \mathbf E(\hat L_n^{\alpha,1}) = \sum_{m < k \le n} \frac{2}{k(k-1)}\cdot \frac{k(k-1)}{n-1} \ . \]
This gives the first claim. Next, letting
\[ E_n := \mathbf E(L^{\alpha,1}_n \mid V_0, \ldots, V_n) = \sum_{m < k \le n}  \frac{k-V_k}{{k \choose 2}} \ ,\]
we have
\[\mathbf {Var}(L^{\alpha,1}_n )= \mathbf {Var}(L^{\alpha,1}_n -E_n)+ \mathbf {Var}(E_n) \ . \]
Now, using Proposition~\ref{P6}, 
\begin{align*}
\mathbf {Var}(&L^{\alpha,1}_n -E_n) =  \sum_{m < k \le n}\mathbf E\Bigl( \Bigl(T_{k-1}-T_k -\frac 1{{k\choose 2}}\Bigr)^2\Bigr)\mathbf E((k-V_k)^2)\\
&=  \sum_{m < k \le n} \frac 1{{k\choose 2}^2} \Bigl( \frac{k^2(k-1)^2}{(n-1)^2} + \frac{k(k-1)(n-k)(n-k-1)}{(n-1)^2(n-2)} \Bigr)\\
&=4\frac{n-m}{(n-1)^2} + 4 \sum_{m < k \le n} \frac{(n-k)(n-k-1)}{k(k-1)(n-1)^2(n-2)} 
\end{align*}
and
\begin{align*}&\mathbf {Var}(E_n) =   \sum_{m < k,l \le n} \frac{1}{{k \choose 2}{l \choose 2}} \mathbf {Cov}(V_k,V_l) \\
&\ = 4\sum_{m < k \le n} \frac{(n-k)(n-k-1)}{k(k-1)(n-1)^2(n-2)}+ 8\sum_{m< k < l \le n} \frac{(n-l)(n-l-1)}{l(l-1)(n-1)^2(n-2)}  \\
&\ = 4\sum_{m < k \le n} \frac{(n-k)(n-k-1)}{k(k-1)(n-1)^2(n-2)}+8 \sum_{m < l \le n} \frac{(l-m-1)(n-l)(n-l-1)}{l(l-1)(n-1)^2(n-2)} \ .
\end{align*}
Thus
\begin{align*}
\mathbf {Var}(L^{\alpha,1}_n )= 4\frac{n-m}{(n-1)^2}  +8 \sum_{m < k \le n}
\frac{(k-m)(n-k)(n-k-1)}{k(k-1)(n-1)^2(n-2)}
\ .
\end{align*}
Now
\begin{align*}
(k-m)&(n-k)(n-k-1)\\&=\bigl(k-1-(m-1)\bigr)\bigl(k(k-1)-2(n-1)k+n(n-1)\bigr)\\
&= k(k-1)^2 -(2n+m-3)k(k-1)\\& \qquad \mbox{}+(n+2m-2)(n-1)k-mn(n-1),
\end{align*}
thus
\begin{align*}
\tfrac 12 &(n-m)(n-2)+\sum_{m < k \le n} \frac{(k-m)(n-k)(n-k-1)}{k(k-1)} \\
&= \tfrac 12 (n-m)(n-2)+\tfrac 12(n-m)(n+m-1) - (n-m)(2n+m-3)\\
&\qquad \mbox{}+ (h_{n-1}-h_{m-1})(n+2m-2)(n-1)- \Bigl(\frac 1m - \frac 1n\Bigr)mn(n-1)\\
&= (h_{n-1}-h_{m-1})(n+2m-2)(n-1)-\tfrac 12 (n-m)(4n+m-5)
\ .
\end{align*}
Combining our formulas the result follows. \qed

\section{The length of a random external branch}\label{S6}

Finally we look at the distribution of the length of an external branch chosen at random. Equivalently, letting $\rho:=\rho(1)$, we may consider 
\[ R_n := T_{\rho} \ , \]
the  length of the branch ending in the leaf with label 1. Its asymptotic
distribution can be obtained in an elementary manner and without recourse to
the results of the preceding sections. Recall 
$ \rho := \max \{k\ge 1 : \{1\} \notin \pi_{k} \}$, 
thus 
\[ \mathbf P(\rho < k ) = \frac{{n-1\choose 2}}{{n \choose 2}} \cdots \frac {{k\choose 2}}{{k+1 \choose 2}} = \frac{k(k-1)}{n(n-1)} \ . \]
Letting
\[ R_n' 
:= \sum_{k=\rho+1}^{n} \frac 1{{k\choose 2}} 
= 2\Bigl(\frac 1\rho - \frac 1n\Bigr) \ , \]
$\{ R_n' > r\} = \{ \rho < 2n/(nr+2)\}$ and for $x>0$ 
\[ \mathbf P( nR_n' > x ) = \mathbf P(\rho < 2n/(x+2)) \sim \frac 4{(x+2)^2} \ . \]
We show that this limiting result carries over to $R_n$. From
\[ R_n-R_n' = \sum_{k=\rho+1}^n \Bigl( T_{k-1}-T_k - \frac 1{{k\choose 2}} \Bigr)= \sum_{k=2}^n \Bigl( T_{k-1}-T_k - \frac 1{{k\choose 2}} \Bigr)I_{\{\rho < k\}} \]
it follows that
\[ \mathbf P\Bigl(R_n-R_n' \neq \sum_{\sqrt n < k \le n} \Bigl( T_{k-1}-T_k - \frac 1{{k\choose 2}} \Bigr)I_{\{\rho < k\}} \Bigr) \le \mathbf P(\rho < \sqrt n) = o(1) . \]
Also from independence
\begin{align*}
\mathbf E\Big[ \Bigl(\sum_{\sqrt n < k \le n} &\Bigl( T_{k-1}-T_k - \frac 1{{k\choose 2}} \Bigr)I_{\{\rho < k\}} \Bigr)^2\Big] \\ &= \sum_{\sqrt n < k \le n} \mathbf E\Big[ \Bigl( T_{k-1}-T_k - \frac 1{{k\choose 2}} \Bigr)^2\Big]\mathbf P(\rho < k) \\
&= \sum_{\sqrt n < k \le n} \frac 1{{k \choose 2}^2} \frac {{k\choose 2}}{{n\choose 2}} = o(n^{-2}) \ . 
\end{align*}
Consequently $R_n= R_n' + o(n^{-1})$ in probability. Thus we end up with the
following result, which was obtained by Caliebe et al (2007) by means of
Laplace transform methods. 

\begin{proposition}\label{P9}
$nR_n$ converges in distribution to the law $\mu$ on $\mathbb R^+$ with density
 $\mu(dx)=8(x+2)^{-3} \, dx$.   
\end{proposition}

\end{document}